\documentclass[12pt,letterpaper]{article}

\newcommand{\figdir}{figures/}

\usepackage{amsmath}
\usepackage{amssymb}
\usepackage{amsthm}
\usepackage[small,margin=3mm]{caption}
\usepackage{dsfont}
\usepackage[hmargin=30mm,top=32mm,bottom=40mm]{geometry}
\usepackage{graphicx}
\usepackage[latin1]{inputenc}
\usepackage{psfrag}

\usepackage{chngcntr}
\counterwithin{figure}{section}
\counterwithin{equation}{section}

\newcommand{\E}{\mathbb{E}}
\newcommand{\Pb}{\mathbb{P}}
\newcommand{\N}{\mathbb{N}}
\newcommand{\Z}{\mathbb{Z}}
\newcommand{\R}{\mathbb{R}}
\newcommand{\I}{\mathds 1}

\theoremstyle{plain}
\newtheorem{theo}{Theorem}[section]
\newtheorem{lemma}{Lemma}[section]

\newtheorem{cor}[lemma]{Corollary}

\theoremstyle{definition}

\newtheorem*{claim*}{Claim}
\theoremstyle{remark}

\newtheorem{question}{Question}

\begin{document}

\title{Fixation for Distributed Clustering Processes}
\author{M. R. Hilário\footnote{
  Instituto de Matemática Pura e Aplicada, Rio de Janeiro, Brazil}, 
  O. Louidor\footnote{
  Courant Institute of Mathematical Sciences, New York
                   University, New York, NY USA},
  C. M. Newman\footnotemark[2],\\
  L. T. Rolla\footnote{Département de Mathématiques et Applications,
  École Normale Supérieure, Paris, France},
  S. Sheffield\footnote{Department of Mathematics, Massachusetts Institute of
                   Technology, Cambridge, MA USA},
  V. Sidoravicius\footnote{
  Centrum voor Wiskunde en Informatica, Amsterdam, Netherlands}
  \ \footnotemark[1]}
\date{January 13, 2010}
\maketitle
\begin{abstract}
We study a discrete-time resource flow in $\mathbb{Z}^d$, where wealthier
vertices attract the resources of their less rich neighbors.
For any translation-invariant probability distribution of initial resource
quantities, we prove that the flow at each vertex terminates after finitely many
steps.
This answers (a generalized version of) a question posed by van den Berg and Meester in 1991.
The proof uses the mass-transport principle and extends to other graphs.
\end{abstract}

This preprint has the same numbering for sections, theorems, equations and figures as the published
article \emph{``Comm. Pure Appl. Math. 63 (2010): 926--934.''}

{\small

}

\section{The model and results}

We consider the following model for distributed clustering.
Initially each vertex $x\in\Z^d$ is assigned a random amount of resource, $0
\leqslant C_0(x) \leqslant \infty$, sampled according to a translation-invariant
distribution.
We denote $x \sim y$ if $x,y$ are adjacent on $\Z^d$, write $x \simeq
y$ if $x=y$ or $x\sim y$, and take $\mathcal G_x=\{y:y\simeq x\}$.
At step $n=0,1,2,\dots$, each vertex $x$ holding some amount of resource will
transfer all its resource to the vertex $a_n(x)\in\mathcal G_x$
with the maximal amount of resource.
All the vertices update simultaneously, leading to the state
$\big(C_{n+1}(x),x\in\Z^d\big)$, where $C_{n+1}(x)$ is the sum of
resources transferred to $x$ at step $n$.
More precisely, we take $a_n(x)=\mathrm{argmax} \{C_n(y),y\in \mathcal G_x\}$ if
$C_n(x)>0$, and otherwise $a_n(x)=x$.
If there is more than one $y\in \mathcal G_x$ that attains the maximum, we say
there is a \emph{tie} at $x$ and in this case $a_n(x)$ is chosen uniformly at
random among the maximizing vertices.
Finally we take $E_n(y) = \{x\in \mathcal G_y:a_n(x)=y\}$ and $C_{n+1}(y) =
\sum_{x \in E_n(y)} C_n(x)$.

Note that, except for the possible tie breaking during the dynamics, all the
randomness is contained in the initial data.
Let $\Pb$ and $\E$ denote the underlying probability measure and its
expectation, for both the initial resource quantities and possible tie breaks.

This model is a simple example of a self-organized structure that emerges from a
disordered initial state.
Instances of this type of phenomenon in several fields of science are
mentioned in~\cite{coffman91}, where the model was introduced.

We are interested in the following phenomena, concerning the stability
properties of this dynamics.
\begin{question}
 \label{item3stability}
 Does each vertex transfer its resource eventually to the same fixed vertex?
\end{question}
\begin{question}
\label{item1fixation}
 Does the the flow at each vertex terminate after finitely many steps?
\end{question}
\begin{question}
 \label{item2conservation}
 If the answer to the previous question is affirmative, is the expected value of
the final resource equal to the expected value of the initial resource?
\end{question}

Van den Berg and Meester~\cite{berg91} considered this model on $\Z^2$ with
continuously-distributed i.i.d.\ initial resource quantities 
and answered Question~\ref{item3stability}.
Namely, they showed that $\Pb[a_n(x) \mbox{ is eventually constant}]=1$.
For the case of i.i.d.\ initial distributions supported on $\N$
(see~\cite{berg91} for the precise hypotheses), they also answered
Question~\ref{item1fixation} in any dimension, i.e., they proved that $\Pb[a_n(x)=x\mbox{
eventually}]=1$.

Again for i.i.d. continuous initial distributions on the two-dimensional
lattice, van den Berg and Ermakov~\cite{berg98} considered some percolative
properties of the configuration after one step, and reduced
Question~\ref{item2conservation} to a finite computation.
As the calculation would be too big, even for the most powerful computers, they
performed Monte Carlo simulations, obtaining overwhelming evidence that
the answer to this question is positive.

%By performing Monte Carlo simulations, they literally proved that the answer to
%this question is positive with high probability.

In this paper we answer Question~\ref{item1fixation} in a rather general
setting.
We consider any dimension and allow any translation-invariant initial
distribution.
Question~\ref{item2conservation} remains open.

\begin{theo}
\label{theo1}
On $\Z^d$, for any translation-invariant distribution of the initial resources,
the flow at each vertex almost surely stops after finitely many steps.
\end{theo}

The proof of local fixation essentially consists of ruling out, one by one, all
the possibilities which could lead to a different scenario.
For this goal we use the mass-transport principle in different ways.

In Section~\ref{sec:lemmas} we prove Theorem~\ref{theo1}.
In Section~\ref{sec:conclusion} we
discuss generalizations of our results
and
conclude with a remark about Question~\ref{item2conservation}.

\section{Proof of Theorem~\ref{theo1}}
\label{sec:lemmas}

Theorem~\ref{theo1} will be proved using a combination of lemmas.
We begin by introducing some extra notation.

Let $|\cdot|$ denote the cardinality of a set.
For each $x \in \Z^d$ and $n=0,1,2,\dots$, one and only one of the following
events will happen (see Figure~\ref{fig1}):
${\cal A}_n(x) = \big[C_n(x)=0\big]$,
${\cal B}_n(x) = \big[C_n(x)>0, E_n(x) = \{x\}\big]$,
${\cal C}_n(x) = \big[E_n(x) = \{z\} \mbox{ for some } z\ne x\big]$,
${\cal D}_n(x) = \big[|E_n(x)| > 1\big]$,
${\cal E}_n(x) = \big[E_n(x) = \emptyset \big]$.
For   $n=0, 1, 2, \dots$ we also set 
\begin{align}
C'_n(x) & := C_n\big(a_n(x)\big), 
\notag \\
L_{n+1}(v) & := a_{n}\big(L_{n}(v)\big),  \mbox{ with }
L_0(v) = v,
\notag \\
S_n(w) &: = \{v:L_n(v)=w\}.  \notag
\end{align}

Observe that $C'_n(x) \geqslant C_n(x)$.
$L_{n}(v)$ is the location at time $n$ of
the resource that initially started at vertex $v$, and $S_n(w)$
denotes the set of all vertices whose initial resource
is located at vertex $w$ at time $n$.

\begin{figure}[!htb]
 \small
 \psfrag{w}{$w$}
 \psfrag{x}{$x$}
 \psfrag{y}{$y$}
 \psfrag{z}{$z$}
 \psfrag{E(w)}{$E(w)$}
 \psfrag{C(x)}{$C(x)$}
 \psfrag{C'(x)}{$C'(x)$}
 \psfrag{C(y)=C'(y)=0}{$C(y)=C'(y)=0$}
 \psfrag{z=a(x)}{$z=a(x)$}
 \psfrag{Aa}{$\mathcal A$}
 \psfrag{Bb}{$\mathcal B$}
 \psfrag{Cc}{$\mathcal C$}
 \psfrag{Dd}{$\mathcal D$}
 \psfrag{Ee}{$\mathcal E$}
 \psfrag{S(x)}{$S(x)$}
 \psfrag{x=L(v)}{$x=L(v)$}
 \psfrag{v in S(x)}{$v \in S(x)$}
 \centering
 \includegraphics[width=13cm]{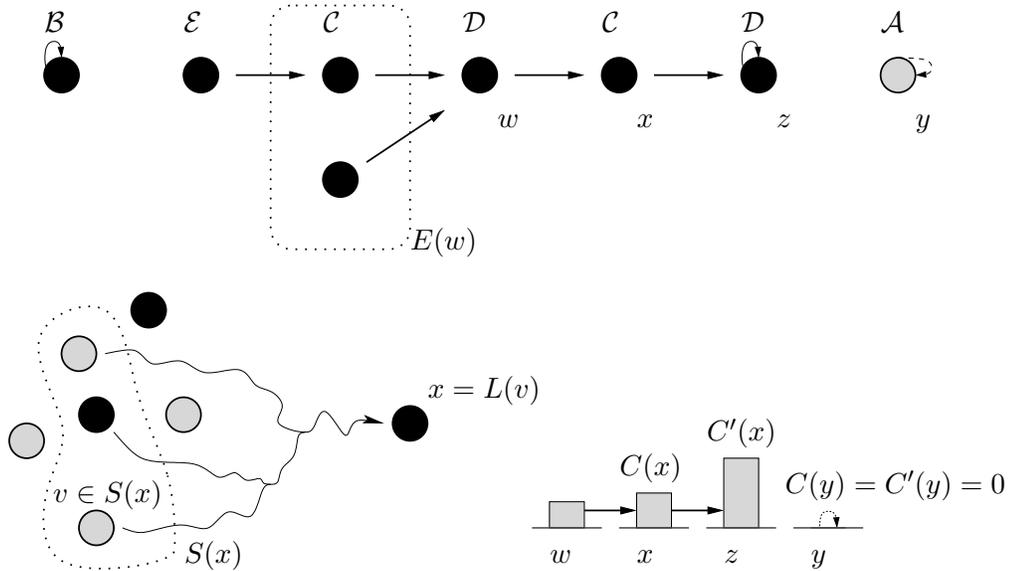}
 \caption{
 Schematic diagram for the notation used in this paper, omitting the index~$n$.}
 \label{fig1}
\end{figure}

\begin{lemma}[\cite{berg91}, p.337]
\label{lemma1}
For each $x$, $\Pb\big({\cal D}_n(x)\mbox{ i.o.}\big)=\Pb\big({\cal
E}_n(x)\mbox{ i.o.}\big)=0$.
\end{lemma}
\begin{proof}
We repeat the proof of \cite{berg91} for the convenience of the reader.

First, $\mathcal E_n(x)$ can happen for at most one value of $n$, after which
one always has $\mathcal A_n(x)$.
By translation-invariance one has
$
 \E \big( |E_n(x)| \big) = \sum_{w} \Pb \big[a_n(x+w)=x\big] =
\sum_{w} \Pb \big[a_n(x)=x-w\big]=1.
$
This is the most common use of the mass transport principle.
Defining $D_n(x) = |E_n(x)|-1$ gives
$
 \E \big[ D_n(x) \big] = 0
$.
So $\Pb \big[ D_n(x) >0 \big] 
 \leqslant \E \big[ D_n(x)^+ \big] = \E \big[ D_n(x)^- \big] = \Pb
\big[ D_n(x) = -1 \big]$.
But the last event corresponds to $\mathcal E_n(x)$, therefore $\sum_n \Pb\big[
D_n(x) >0\big] \leqslant 1$.
Since $\big[D_n(x) >0\big]$ corresponds to $\mathcal D_n(x)$, the result follows
by the Borel-Cantelli Lemma.
\end{proof}

Note that $C_{n+1}(x)>C_{n}(x)$ can only happen on the event ${\cal D}_n(x)$.
So, almost surely, $\big(C_{n}(x)\big)$ is a nonnegative, eventually
non-increasing sequence, and therefore $\lim_{n \to \infty} C_n(x)$ exists.

The following lemma permits us to extend the results of~\cite{berg91} to a
general translation-invariant probability distribution for the initial resource
quantities.
\begin{lemma}
\label{lemma2}
$\Pb$-a.s., ties cannot happen infinitely often at a fixed vertex.
\end{lemma}
\begin{proof}
For a pair of vertices $y\simeq x$, let
\[
 {\cal T}_n(x,y)=\left[0<C_n(x)=C_n(z)=\max\big\{C_n(w),w\simeq y\big\}
 \mbox{ for some } z\simeq y, z\ne x\right]
\]
denote the event that there is a tie at $y$ at time $n$, and $x$ is one of
the candidates for its resource.
Denote by ${\cal F}_{y,n}$ the $\sigma$-field generated by
\[
\big(C_m(z),m\leqslant n \mbox{ and }z\in\Z^d\big) \mbox{ and }\big(a_m(z),m<n \mbox{ and }z\in\Z^d \mbox{ or }
m=n \mbox{ and } z\ne y\big),
\]
i.e., ${\cal F}_{y,n}$ contains all the information up to step $n$, except for
the possible tie breaking at vertex $y$ at the $n$-th step.
Write $M_{x\backslash y}^n = \big| E_n(x)\backslash \{y\} \big|$.
Note that $M_{x\backslash y}^n$ and $\mathcal T_n(x,y)$ are ${\cal
F}_{y,n}$-measurable.
Conditioning on ${\cal F}_{y,n}$, if $M_{x\backslash y}^n = 0$ and ${\cal
T}_n(x,y)$ occurs, then with probability at least $\frac12$, $a_n(y)\ne x$, in
which case ${\cal E}_n(x)$ will happen and thus $C_m(x)=0$ for all $m>n$
(therefore ${\cal T}_n(x,y)$ can never happen again).
If $M_{x\backslash y}^n > 0$ and ${\cal T}_n(x,y)$ occurs, then with probability
at least $\frac1{2d+1}$, $a_n(y) = x$, in which case ${\cal D}_n(x)$ will
happen.
Therefore, the occurrence of ${\cal T}_n(x,y)$ for infinitely many $n$'s implies
almost surely the occurrence of ${\cal D}_n(x)$ infinitely often.
But the latter event has probability 0 by Lemma~\ref{lemma1}.
\end{proof}

\begin{lemma}[\cite{berg91}, p.338]
\label{lemma3}
$\Pb$-a.s. for each vertex $x$, only one of the events $\limsup_n{\cal A}_n(x)$,
$\limsup_n{\cal B}_n(x)$ or $\limsup_n{\cal C}_n(x)$ will happen.
\end{lemma}
\begin{proof}
Once we know that there are finitely many ties at most, the argument of
\cite{berg91} can be applied.
We present the proof for the convenience of the reader.

By Lemmas~\ref{lemma1} and~\ref{lemma2} we can take $n_0$ so that neither ${\cal
D}_n(x)$ nor ${\cal E}_n(x)$ will happen, nor will there be a tie at any $y
\in\mathcal G_x$, for any $n \geqslant n_0$.
If ${\cal A}_n(x)$ happens for some $n$, it also happens for all $m>n$.
It is thus enough to prove that, if ${\cal C}_n(x)$ happens for some $n
\geqslant n_0$, then ${\cal C}_{n+1}(x)$ also happens.
Now, if $\mathcal C_n(x)$ happens, we have $C_{n+1}(z) \geqslant C_n(x) >
C_n(y) = C_{n+1}(x)$, where $E_n(x) = \{y\}$ and $z = a_n(x)$.
The strict inequality holds because there cannot be a tie at $y$.
Therefore $a_{n+1}(x) \neq x$ and, since ${\cal D}$ and ${\cal E}$ have been
ruled out, we must have ${\cal C}_{n+1}(x)$ again.
\end{proof}

By Lemma~\ref{lemma3} we can say that each vertex is 
uniquely either an ${\cal A}$-vertex or a
${\cal B}$-vertex or a ${\cal C}$-vertex.

\begin{cor}
\label{cor3}
$\Pb$-a.s., $C'_n(x) - C_n(x) \rightarrow 0$ for all $x\in\Z^d$.
\end{cor}
\begin{proof}
If $x$ is an ${\cal A}$-vertex or a ${\cal B}$-vertex then $C'_n(x) - C_n(x) =0$
eventually.
So suppose $x$ is a ${\cal C}$-vertex.
Take $n_0$ so that ${\cal C}_n(x)$ happens for all for $n\geqslant n_0$.
By Lemma~\ref{lemma1} we can further assume that $\mathcal D_n(z)$ will not
happen for any $z\in\mathcal G_x$ and $n\geqslant n_0$, in particular $C_{n+1}(z)\leqslant C_n(z)$.
We claim that $C'_{n+2d}(x)\leqslant C_n(x)$ for all $n\geqslant n_0$.
Since in addition $C_{n+2d}(x) \leqslant C'_{n+2d}(x)$  and $C_n(x)$
converges, this will finish the proof.

Let $n \geqslant n_0$ be fixed and for $m\geqslant n$ let $V_m = \big\{y\sim x:
C_m(y)>C_n(x)\big\}$.
Since resource quantities no longer increase, $V_{m+1}\subseteq V_m$.
Moreover, the occurrence of $\mathcal C_n(x)$ implies that $a_n(x)\neq x$
and, since $\mathcal D_n\big(a_n(x)\big)$ cannot occur, $a_n(x)$ is a ${\cal
C}$-vertex and $E_n\big(a_n(x)\big)=\{x\}$.
Thus, $C_{n+1}\big(a_n(x)\big)=C_n(x)$ and therefore $a_n(x)\not\in V_{n+1}$.
If $V_{n+1}=\emptyset$ we are done, and if $V_{n+1}\ne\emptyset$ we find again
that $a_{n+1}(x)\in V_{n+1}\backslash V_{n+2}$.
Proceeding this way, and since $|V_n|<2d$, we must find $V_{n+j}=\emptyset$ in
less than $2d$ steps.
\end{proof}

\begin{lemma}
\label{lemma4}
For each fixed $k\in\N$ and $x\in\Z^d$,
\[
 \Pb\left( 0 <C'_n(x)-C_n(x) < \delta,\ \big|S_n(x)\big|\leqslant k \right)
 \to 0
\]
as $\delta\to0$, uniformly in $n$.
\end{lemma}
\begin{proof}
For $v\in\Z^d$ consider $S_n\big(L_n(v)\big)$, the set of vertices
whose initial resource were joined with that of $v$ by time $n$ (according
to the dynamics rules, once two or more initially distinct 
sources join, they never separate), and define $N_n(v) := \left|S_n\big(L_n(v)\big)\right|$.
For each fixed $v$, $N_n(v)$ is non-decreasing in $n$.

Let $A_n(v) = C_n\big(L_n(v)\big)$ be the total amount of resource at $L_n(v)$
at
time $n$ and let $A'_n(v) = C'_n\big(L_n(v)\big) = C_n\big(L_{n+1}(v)\big)$.
For each vertex $v$, the value $A'_n(v) - A_n(v)$ is nonnegative
and it can only decrease when $N_n(v)$ increases.

It is therefore the case that $\inf\big[A'_n(v)-A_n(v)\big]$, the infimum taken
over all $n$ such that $N_n(v)\leqslant k$ and $A'_n(v)>A_n(v)$, is a strictly
positive random variable.
In particular, for fixed $k$,
\begin{equation}
 \label{eq1uniformdecay}
  \Pb\big(0<A'_n(v)-A_n(v)<\delta, N_n(v) \leqslant
  k\big)\to0 \mbox{ as } \delta \to 0 \quad \mbox{ uniformly in $n$}
 .
\end{equation}

We shall relate the above limit to the desired result via a quantitative use of
the mass-transport principle, what we call an unlikelihood transfer argument.

Let us omit $\delta,k$ from now on and denote the event in~\eqref{eq1uniformdecay}
by $\mathcal U_{n}(v)$.
For $x\in\Z^d$, let $\mathcal V_{n}(x)$ denote the event that $\mathcal
U_{n}(v)$ occurs for some $v\in S_n(x)$.
If $\mathcal U_{n}(v)$ happens for any such $v$ then it happens for all of them
(because the values of $A_n$, $A_n'$ and $N_n$ are constant within $S_n(x)$ and
equal $C_n(x)$, $C_n'(x)$, and $|S_n(x)|$, respectively).
Writing $m(v,x)=\I_{L_n(v)=x,\mathcal U_n(v)}$ we have that $\I_{\mathcal
U_n(v)}=\sum_x m(v,x)$.
It follows from the translation invariance of the process that
$
 \E \sum_y m(w,w+y) =
 \E \sum_y m(w-y,w) 
$,
which means
\begin{equation}
 \label{eq1transport}
 \textstyle
 \E \sum_x m(w,x) =
 \E \sum_v m(v,w)
 \qquad \forall\ w
 ,
\end{equation}
giving
\begin{equation}
 \label{eq1unlikely}
 \Pb\big[\mathcal U_{n}(v)\big] = \E \big[ |S_n(x)| \I_{\mathcal
 V_{n}(x)} \big]
 \geqslant  \Pb\big(\mathcal V_{n}(x)\big)
 \qquad \forall\ v,x
 .
\end{equation}

Now $\mathcal V_{n}(x)$ is exactly the event considered in the statement of the
lemma and since $\Pb\big[\mathcal U_{n}(v)\big] \to 0$  as $\delta \to 0$
uniformly in $n$ the result follows.
\end{proof}

\begin{proof}[Proof of Theorem~\ref{theo1}:]
By Lemma~\ref{lemma3} it is enough to show that $\Pb\big(a_n(x)\ne x\big)\to 0$
as $n\to\infty$.
For any choice of $\delta>0$, if $a_n(x)\ne x $ then either
$C'_n(x)=C_n(x)$, which implies that there is a tie at $x$, or
$C'_n(x)-C_n(x)\geqslant \delta$ or $0 < C'_n(x) - C_n(x) < \delta$.
By Lemma~\ref{lemma2} and Corollary~\ref{cor3} the probabilities of the first
two events tend to zero as $n\to\infty$.

For any choice of $k\in\N$, the last event can be split into two cases,
according to whether $|S_n(x)| \leqslant k$ or not.
By Lemma~\ref{lemma4}, the probability that this event happens in the first
case tends to zero as $\delta\to0$, uniformly in $n$.
Now, with $m(v,x)=\I_{L_n(v)=x}$, \eqref{eq1transport} gives
$\E\big(|S_n(x)|\big)=1$ so $\Pb\big(|S_n(x)|>k\big)<\frac1k$.

So, if we consider $\limsup_{n\to\infty}\Pb\big(a_n(x)\ne x\big)$,
let $\delta\to0$ and then let $k\to\infty$ we get the desired limit.
\end{proof}

\section{Concluding remarks}
\label{sec:conclusion}

We conclude this paper by
discussing how Theorem~\ref{theo1} extends to more
general settings and why a positive
answer to Question~\ref{item2conservation} does not follow from the previous
arguments.

\paragraph{Generalizations}

Our proof applies in other settings with much generality, as long as the mass transport principle~\eqref{eq1transport} is true.
It thus covers cases when the distribution of initial resources is invariant with respect to a transitive unimodular group of automorphisms.
Examples include Cayley graphs, regular trees, etc.
The only change in the proof is to replace $2d$ by the graph degree.

It also covers graphs that are locally finite and can be periodically embedded in $\R^d$.
Namely, one can consider graphs whose vertex set may be written as
$[J]\times \Z^d$, where $[J]:=\{1,\dots,J\}$,
and whose edge set is invariant under the mappings $(j,x)\mapsto(j,x+y)$ for all
$y\in\Z^d$.
Here translation invariance is understood as the distribution of the
initial resources being invariant under the above mappings.

The changes in the proof for the above case are the following.
In Lemma~\ref{lemma1}, notice that $\mathcal E_n(j,x)$ will happen at most once
for each $j\in[J]$ and fixed $x$; so $D_n(x): = \big(\sum_1^J|E_n(j,x)|\big)-J$ 
satisfies $\E\big[\sum_n D_n(x)^+\big] = \E\big[\sum_n D_n(x)^-\big] \leqslant J$,
thus
$D_n(x)>0$ can occur for at most finitely many $n$'s;
finally $\mathcal D_n(j,x)$ corresponds to $|E_n(j,x)|-1>0$, which in turn
implies that $D_n(x)>0$ or that
$\mathcal E_n(j',x)$ occurs for some other $j'\in[J]$.
In Lemma~\ref{lemma2} and Corollary~\ref{cor3} replace $2d$ by $\max_j
\deg(j,x)$.
In Lemma~\ref{lemma4} we take $m(v,x)=\sum_{j,j'}\I_{L_n(j,v)=(j',x),\mathcal
U_n(j,v)}$ and the mass-transport principle~\eqref{eq1transport} holds with the
same proof.
In the proof of Theorem~\ref{theo1} we can write $\E\big(\sum_j
|S_n(j,x)|\big)=J$, giving $\Pb\big(|S_n(j,x)|>k\big) < \frac J k$ for any
$(j,x)$.

\paragraph{Open question}

It would be natural to try to answer Question~\ref{item2conservation} using an argument similar to the proof of Theorem~\ref{theo1}, in a way that uses the result to reinforce itself.

First notice that a yes to Question~\ref{item2conservation} is equivalent to showing that every amount of initial resource eventually stops moving.
Indeed, if $C_\infty(x):=\lim_n C_n(x)$ and $F_v$ denotes the event that $L_n(v)$ is eventually constant, then $\E\big[C_\infty(x)\big]= \E\big[C_0(v)\I_{F_v}\big]$, so a yes to the question is equivalent to $F_v$ happening almost surely.
This equality follows from Theorem~\ref{theo1} and~\eqref{eq1transport} with
\[
 m(v,x) =
 \begin{cases}
 C_0(v), & \mbox{if } L_n(v)=x \mbox{ eventually,} \\
 0, & \mbox{otherwise.}
 \end{cases}
\]

Let us denote by $\tilde{\mathcal U}_n(v)$ the event that, for some $m>n$,
$L_{m+1}(v)\not =L_m(v)$.
This event means that the resource that started at $v$ is still going to move
after time $n$.
Therefore Question~\ref{item2conservation} boils down to 
whether we can show that
$\Pb\big[\tilde{\mathcal U}_n(v)\big]\to0$ as $n\to\infty$.

Let $\tilde{\mathcal V}_n(x)$ denote the event that the resource found at vertex
$x$ at time $n$ will leave $x$ after time $n$, that is, $a_m(x)\ne x$ for some
$m>n$.
Note that this is equivalent to the event that $\tilde{\mathcal U}_{n}(v)$
occurs for some $v\in S_n(x)$, and again if it happens for any such $v$ then it
happens for all of them.
This implies that the analogue of~\eqref{eq1unlikely} holds.

By Theorem~\ref{theo1}, $\Pb\big[\tilde{\mathcal V}_n(x)\big]\to0$.
The inequality in the analogue of~\eqref{eq1unlikely} does not help here, but we
could make use of the equality if the distribution of $|S_n(x)|$ satisfied an
appropriate moment (or uniform integrability) condition.
This is another quantitative use of the mass-transport principle, that we call
unlikelihood sharing principle.

For instance, suppose one can prove that $E\big(|S_n(x)|^\alpha)$ stays bounded as $n \to \infty$ for some $\alpha>1$.
Then $\Pb\big[\tilde{\mathcal V}_n(x)\big]\to0$ implies that $\Pb\big[\tilde{\mathcal U}_n(v)\big]\to0$ since
\[
 \E ( S \I_{\tilde{\mathcal V}} )
 = \Pb(\tilde{\mathcal V}) \E ( S |{\tilde{\mathcal V}} )
 \leqslant
 \Pb(\tilde{\mathcal V}) 
 \big[\E ( S^\alpha |{\tilde{\mathcal V}}) \big]^{1/\alpha}
 \leqslant
 \big[\Pb(\tilde{\mathcal V})\big]^{1-\frac1\alpha} \E \big[ S^\alpha
  \big]^{1/\alpha},
\]
where $S=|S_n(x)|$ and $\tilde{\mathcal V}=\tilde{\mathcal V}_{n}(x)$.

Unfortunately, we do not know how to control the tail of the distribution of
$|S_n(x)|$, and Question~\ref{item2conservation} remains an open problem.

While the final version of this paper was in preparation, it was found by van den Berg, Hilário and Holroyd~\cite{bergXX} that Question~\ref{item2conservation} as stated has a negative answer: there do exist nonnegative translation-invariant initial distributions for which resources do escape to infinity.

\paragraph{Acknowledgments}

This work had financial support from ARCUS Program, CNPq grants 140532/2007-2
and 141114/2004-5, FAPERJ, FAPESP grant 07/58470-1, FSM-Paris,
German Israeli Foundation grant I-870-58.6/2005 (O.L.)
US NSF under grants DMS-08-25081 (O.L.), DMS-06-06696 (C.M.N.), DMS-06-45585 (S.S.),
and OISE-07-30136 (O.L., C.M.N., and S.S.).

Part of this collaboration took place during the Fall, 2008 semester on
\emph{Interacting Particle Systems} at IHP, Paris; we thank the organizers and
staff for providing a nice working environment.
M.~R.~Hilário and L.~T.~Rolla thank CWI-Amsterdam and C.~M.~Newman and
L.~T.~Rolla thank CRM-Montréal for hospitality.
V.~Sidoravicius thanks Chris Hoffman for many fruitful discussions.

\end{document}